\documentclass[11pt,reqno]{amsart}
\usepackage{amsmath,amsthm,amsfonts,amssymb}
\usepackage{mathtools}
\usepackage[czech,english]{babel}
\usepackage{graphicx}
\usepackage[all]{xy}
\usepackage{mdwtab}
\usepackage{tensor}
\usepackage[OT1]{fontenc}
\RequirePackage{picinpar}

\newtheorem{theorem}{Theorem}

\newtheorem{lemma}{Lemma}

\newtheorem{corollary}{Corollary}
\newtheorem*{mcorollary}{Corollary}

\newtheorem{definition}{Definition}

\newtheorem{remark}{Remark}
\newcommand{\bbbn}{\mathbb{N}}

\newcommand{\Nat}{{\mathbb N}}

\newcommand{\switch}[3]{{#1}\,?\,{#2}\,:\,{#3}}

\newcommand{\restr}[2]{{#1}\upharpoonright_{#2}}
\title{Limits of Mappings}
\thanks{Supported by grant ERCCZ LL-1201 
and CE-ITI, and by the European Associated Laboratory ``Structures in
Combinatorics'' (LEA STRUCO) P202/12/G061}
\author[L. Hosseini]{Lucas Hosseini}
\address{Lucas Hosseini\\
Centre d'Analyse et de Math\'ematiques Sociales (CNRS, UMR 8557)\\
  190-198 avenue de France, 75013 Paris, France}
  \email{lucas.hosseini@gmail.com}
\author[J. Ne{\v s}et{\v r}il]{Jaroslav Ne{\v s}et{\v r}il}
\address{Jaroslav Ne{\v s}et{\v r}il\\
Computer Science Institute of Charles University (IUUK and ITI)\\
   Malostransk\' e n\' am.25, 11800 Praha 1, Czech Republic}
\email{nesetril@kam.ms.mff.cuni.cz}
\author[P. Ossona de Mendez]{Patrice~Ossona~de~Mendez}
\address{Patrice~Ossona~de~Mendez\\
Centre d'Analyse et de Math\'ematiques Sociales (CNRS, UMR 8557)\\
  190-198 avenue de France, 75013 Paris, France
  and
     Computer Science Institute of Charles University (IUUK)\\
   Malostransk\' e n\' am.25, 11800 Praha 1, Czech Republic}
 \email{pom@ehess.fr}
\begin{document}
\maketitle

\footnotetext{\begin{window}[0,r,{\includegraphics[height=7mm]{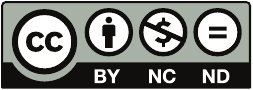}},{}]
    \noindent This work has been submitted to the {\em European Journal of Combinatorics} and 
    is licensed under the Creative Commons Attribution-NonCommercial-NoDerivatives 
    4.0 International License. 
    To view a copy of this license, visit\\
    {\tt creativecommons.org/licenses/by-nc-nd/4.0}.
  \end{window}
  }
  
 \begin{abstract}
 In this paper we consider a simple algebraic structure --- sets with a single endofunction. We shall see that from the point of view of structural limits, even this simplest case is both interesting and difficult. Nevertheless we obtain the shape of limit objects in the full generality, and we prove the inverse theorem in the case of quantifier-free limits.
\end{abstract}

\section{Introduction}
The aim of this paper is to construct analytic limit objects for convergent sequences of finite mappings $f_n:F_n\rightarrow F_n$ (``finite'' meaning that the sets $F_n$ are finite) and, conversely, to approximate a limit object by a finite mapping. This work originated within the scope of the recent studies of graph limits \cite{LovaszBook}, and more precisely within the framework of structural limits \cite{CMUC}. 
As this framework is closely related to finite model theory, instead of describing mappings as $f:F\rightarrow F$ we shall define mappings as structures $\mathbf F$ (boldface) with signature $\{f\}$, where $f$ is a unary function symbol, with domain $F$ (same symbol as $\mathbf F$ but not boldface) and with interpretation of $f$ denoted by  $f_{\mathbf F}$. Hence $f_{\mathbf F}:F\rightarrow F$ and, for $u,v\in F$ we have the two following possible writings for the property that $v$ is the image of $u$: either $f_{\mathbf F}(u)=v$ or
$\mathbf F\models (f(u)=v)$.

In the general framework introduced in  \cite{CMUC}, the notion of convergence of structures is conceptualized 
by means of the convergence of the satisfaction probability of formulas in a fixed fragment of first-order logic: A sequence $(\mathbf A_n)_{n\in\bbbn}$ of finite structures is {\em $X$-convergent} (where $X$ is a given fragment of first-order logic) if, for every formula $\phi\in X$ the probability $\langle \phi,\mathbf A_n\rangle$ of satisfaction of $\phi$ in $\mathbf A_n$ for a random (independent uniform) assignment of elements of the domain $A_n$ of $\mathbf A_n$ to the free variables of $\phi$ converges as $n$ grows to infinity.  
Three fragments of first-order logic, defining a gradation of three notions of convergence, will be of special interest: the fragment QF of quantifier-free formulas, the fragment ${\rm FO}^{\rm local}$ of {\em local} formulas (that is of formulas whose satisfaction only depends on a fixed neighborhood of the free variables) and the full FO fragment of all first-order formulas.

This framework allows one to consider limits of general combinatorial structures, and was  applied to the study of limits of sparse graphs \cite{gajarsky2016first, modeling_arxiv,modeling,  limit1}, 
   matroids \cite{Kardos2017150}, and  tree semi-lattices \cite{QFTSL-arxiv}. 
It is sometimes possible (although this is not the case in general  \cite{limit1}) to represent the limit by a particularly nice analytic object, called {\em modeling}, which is a structure whose domain is a standard Borel space endowed with a Borel probability measure, with the property that every (first-order) definable set is Borel measurable. 


In order to make the motivation of this paper clear, we take time
in Section~\ref{sec:context} for  a quick review of some of the fundamental notions and problems encountered in the domain of graph limits, and how they are related to the study of limits and approximations of algebras (that is of functional structures).  

The first main result of this paper is the construction, for every  ${\rm FO}$-convergent sequence of finite mappings, of a modeling representing the FO-limit of the sequence (Theorem~\ref{thm:folimmap}). As every sequence of finite structures contains an FO-convergent subsequence \cite{CMUC}, this modeling can be used to represent the limits for the (weaker notions of) ${\rm FO}^{\rm local}$-convergence and QF-convergence.
Theorem~\ref{thm:folimmap} is proved as a combination of general results about limit distributions (stated below as Theorem~\ref{thm:cmuc}, see \cite{CMUC}) and methods developed in \cite{modeling} for the construction of modeling limits of trees. As a consequence we are able to deduce the form of limits of mappings.

We shall also be interested in the inverse problems, which aim to determine which objects are $X$-limits of finite mappings (for a given fragment $X$ of first-order logic). It should be noticed that although the inverse problem for QF-limits of graphs or hypergraphs has been completely solved, the inverse problem for ${\rm FO}^{\rm local}$-limits of graphs with uniformly bounded degrees (or equivalently for ${\rm QF}$-limits of algebras with $d$ involutions, see Section ~\ref{sec:context}), which is known as Aldous--Lyons conjecture \cite{Aldous2006} is wide open. 

In our (restricted) setting of algebras with a single function symbol, we solve the inverse problem for ${\rm QF}$-limits. 
(Note that solving the inverse problem for ${\rm QF}$-limits of algebras with $2$ function symbols would imply solving the Aldous--Lyons conjecture.)
The solution of the inverse 
problems for ${\rm FO}^{\rm local}$-limits and ${\rm FO}$-limits
of mappings,  stated as Theorems~\ref{thm:invlocmap} and~\ref{thm:invfomap} in Section~\ref{sec:conc} will be proved in a forthcoming paper.

\section{Definitions and Notations}
\label{sec:def}
Recall that  a $\sigma$-structure $\mathbf A$ is defined by its {\em domain}  $A$, its {\em signature} $\sigma$ (which is a set of symbols of relations and functions together with their arities), and the interpretation  all the relations and functions  in $\sigma$ as relations and functions on $A$.

The structures we consider here are structures with signature $\sigma$ consisting of a single functional symbol $f$ and (possibly) some unary symbols $M_1,\dots,M_c$ (interpreted as a coloring). We call such structures {\em colored mappings} (or simply {\em mappings}).

Let $\mathbf F$ be such a mapping (with domain $F$). Then $f_{\mathbf F}$ is the interpretation of the symbol $f$ in $\mathbf F$ (thus $f_{\mathbf F}:F\rightarrow F$). 
For a first-order formula $\phi$ with $p$ free variables and a mapping $\mathbf F$ we define 
$$\phi(\mathbf F)=\{(v_1,\dots,v_p)\in F^p:\ \mathbf F\models\phi(v_1,\dots,v_p)\}.$$

If $\mathbf F$ is finite (meaning that $F$ is finite) we further define
$$
\langle\phi,\mathbf F\rangle=\frac{|\phi(\mathbf F)|}{|F|^p}.
$$
If the domain $F$ of $\mathbf F$ is a standard Borel space equipped with a (Borel) probability measure\footnote{Strictly speaking, this is not the same as a standard probability space (for which the probability measure is required to be complete).} $\nu_{\mathbf F}$ and if $\phi(\mathbf F)$ is a Borel subset of $F$ we define
$$
\langle\phi,\mathbf F\rangle=\nu_{\mathbf F}^{\otimes p}(\phi(\mathbf F)),
$$
where $\nu_{\mathbf F}^{\otimes p}$ stands for the product measure
$\overbrace{\nu_{\mathbf F}\otimes\dots\otimes\nu_{\mathbf F}}^{p  \text{ times}}$ on the power space $F^p$.

Recall that a measurable space $(F, \Sigma)$ is said to be {\em standard Borel} \cite{mackey1957borel} if there exists a metric on $F$ which makes it a complete separable metric space in such a way that $\Sigma$ coincides with the Borel $\sigma$-algebra (that is the smallest $\sigma$-algebra containing the open sets). Several very useful results hold for
standard Borel spaces (but do not hold in general). For instance, any bijective measurable mapping between standard Borel spaces is an isomorphism. In this paper, we only consider standard Borel spaces,  and {\em measurable} will mean {\em Borel measurable}, unless otherwise stated. For instance, the terms {\em measurable function} and {\em probability measure} will refer to {\em Borel measurable function} and {\em Borel  probability measure}. 

The mapping $\mathbf F$ is {\em Borel} if the function
$f_{\mathbf F}:F\rightarrow F$ is Borel (that is Borel measurable);
it is a {\em modeling mapping} if every first-order definable subset $\phi(\mathbf F)$ of a power of $F$ is Borel.
Note that every modeling mapping is Borel, as the graph of $f_{\mathbf F}$ is first-order definable.

Given two elements $u,v$ of $F$, we define the distance ${\rm dist}(x,y)$ as the minimum value $a+b$ such that $a,b\geq 0$ and
$f_{\mathbf F}^a(u)=f_{\mathbf F}^b(v)$. Note that ${\rm dist}(u,v)$ is exactly the graph distance between $u$ and $v$ in the {\em Gaifman graph} of $\mathbf F$ (that is in the graph with vertex set $F$ where two vertices $x,y$ are adjacent if either $x=f_{\mathbf F}(y)$ or $y=f_{\mathbf F}(x)$).
For $u\in F$ and $r\in\bbbn$ we denote by $B_r(\mathbf F,u)$ the {\em $r$-ball} of $u$ in $\mathbf F$, that is the set of all elements of $F$ at distance at most $r$ from $u$.

Let $\mathbf F$ be a modeling mapping and let $X$ be a  Borel subset of $F$ with positive measure.  We denote by $\restr{\mathbf F}{X}$
the restriction of $\mathbf F$, which is the modeling mapping with domain $X$, probability measure $\nu_{\restr{\mathbf F}{X}}=\frac{1}{\nu_{\mathbf F}(X)}\nu_{\mathbf F}$ and
$$
f_{\restr{\mathbf F}{X}}(v)=\begin{cases}
	f_{\mathbf F}(v)&\text{if }f_{\mathbf F}(v)\in X\\
	v&\text{otherwise}
\end{cases}
$$

A formula $\phi$ (with $p$ free variables) is {\em local} if its satisfaction only depends on a fixed $r$-neighborhood of its free variables.

We define the following fragments of the set of all first-order formulas:
\begin{itemize}
	\item ${\rm QF}$, the set of all quantifier free formulas;
	\item ${\rm FO}_0$, the set of all sentences (that is of all formulas without free variables);
	\item ${\rm FO}_1$, the set of all formulas with a single free variable;
	\item ${\rm FO}_1^{\rm local}$, the set of all local formulas with a single free variable;
	\item ${\rm FO}^{\rm local}$, the set of all local formulas;
	\item ${\rm FO}$, the set of all first-order formulas.
\end{itemize}

The idea to conceptualize limits of structures by means of convergence of the satisfaction probability of formulas in a fixed fragment of first-order logic has been introduced by two of the authors in 
\cite{CMUC}.
Given a fragment $X$ of first-order formulas among the above ones,
a sequence $(\mathbf A_n)_{n\in\bbbn}$ of structures is {\em $X$-convergent} (or {\em convergent} when $X$ is clear from the context) if, for every first-order formula $\phi\in X$  the probability $\langle\phi,\mathbf A_n\rangle$ that $\phi$ is satisfied in $\mathbf A_n$ for a random assignment of elements of $\mathbf A_n$ to the free variables of $\phi$ converges as $n$ grows to infinity.  In the particular case where $X={\rm FO}^{\rm local}$ we shall indiscriminately use the terms of ${\rm FO}^{\rm local}$-convergence and of {\em local convergence}.

This framework allows to consider limits of general combinatorial structures, and was applied to limits of sparse graphs with unbounded degrees \cite{limit1,modeling,gajarsky2016first}, 
  matroids \cite{Kardos2017150}, and tree semi lattices \cite{QFTSL-arxiv}. 
%
Moreover, if $\mathbf L$ is a Borel structure such that every $X$-definable subset of a power of $L$ is Borel we say that $\mathbf L$ is a {\em $X$-limit} of the sequence $(\mathbf A_n)_{n\in\bbbn}$ and we denote by
$\mathbf A_n\xrightarrow{X} \mathbf L$ the property that for every first-order formula $\phi\in X$ it holds
$$
\langle\phi,\mathbf L\rangle=\lim_{n\rightarrow\infty}\langle\phi,\mathbf A_n\rangle.
$$

Given a mapping $\mathbf F$ and a non-negative integer $k$, an element $x\in F$ is a {\em  $k$-cyclic element} of $\mathbf F$ if $f_{\mathbf F}^k(x)=x$ but
$f_{\mathbf F}^i(x)\neq x$ for every $0< i<k$; the element $x$ is a {\em cyclic element}  of $\mathbf F$ if it is $k$-cyclic for some non-negative integer $k$. We denote by $Z_k(\mathbf F)$ (resp. $Z(\mathbf F)$) the set of the $k$-cyclic elements (resp. cyclic elements) of $\mathbf F$:
\begin{align*}
Z_k(\mathbf L)&=\{v\in L: f_{\mathbf L}^k(v)=v\}\\
Z(\mathbf F)&=\bigcup_{k\geq 0}Z_k(\mathbf F).
\end{align*}

We consider a mapping modeling $\mathbf F$ with no atoms.
Its domain $F$ is partitioned into the countably many subsets 
$$F_i=\{x\in F: |f_{\mathbf F}^{-1}(x)|=i\}$$
for $i=0,1,\dots,$ and 
$$F_\infty=\{x\in F: |f_{\mathbf F}^{-1}(x)|=\infty\}.$$

\begin{definition}
\label{def:FMTP}
The {\em Finitary Mass Transport Principle} (FMTP) for $\mathbf F$
consists of the following two conditions:
\begin{itemize}
\item $\nu_{\mathbf F}(F_\infty)=0$;
	\item for every measurable subsets $A,B$ of $F\setminus F_\infty$ 
it holds

\begin{equation}
	\nu_{\mathbf F}(A\cap f_{\mathbf F}^{-1}(B))=
	\int_B|f_{\mathbf F}^{-1}(y)\cap A|\,{\rm d}\nu_{\mathbf F}(y)
\end{equation}
	\end{itemize}
\end{definition}
	
Note that a direct consequence of the FMTP is that for every measurable subset $A$ of $F$ it holds $\nu_{\mathbf F}(A)\geq \nu_{\mathbf F}(f_{\mathbf F}(A))$.

\section{Particular Types of Convergence}
\label{sec:context}

The notion of {\em left convergence} of Lov\'asz and Szegedy \cite{Lov'asz2006} is strongly related to important tools of extremal graph (and hypergraph) theory, like the Szemer\'edi regularity lemma and its extensions, or Razborov's flag algebra \cite{Razborov2007}. It is also strongly related to probability theory, and in particular to generalizations of de Finetti's theorem by Hoover \cite{Hoover1979} and Aldous \cite{Aldous1981,Aldous1985,AldousICM} (see also \cite{Austin2013}). In this setting, a convergent sequence of graphs admit as a limit a simple object, called {\em graphon}, which is a Lebesgue measurable function $W:[0,1]\times[0,1]\rightarrow [0,1]$ (defined up to measure preserving transformations). For $k$-regular hypergraphs, the limit object, called {\em hypergraphon} \cite{ElekSze} is a bit more complex (it is a measurable function from $[0,1]^{2^k-2}$ to $[0,1]$). Although this last case extends to  left limits of relational structures \cite{aroskar2014limits},  the case of structures with functional symbols remains open. In the graph (or hypergraph) setting, a sequence is left-convergent if the probability that $k$ random elements induce a specific substructure converges. In our interpretation (which we introduce in detail below), a sequence of structures is left-convergent (or ${\rm QF}$-convergent) if the probability of satisfaction of any fixed quantifier-free first-order formula by a random assignment of the free variables (drawn uniformly and independently at random) converges. (Here ${\rm QF}$ stands for Quantifier-Free.)

In this paper we consider limits of algebras (that is structures with functions but no relations), in the simplest case of a single unary operation. Note that algebras appear to be an unexpected bridge between the theory of left convergence and the theory of {\em local convergence} of bounded degree graphs of Benjamini and Schramm \cite{Benjamini2001}, as we shall see now (see also \cite{QFTSL-arxiv}).
In the setting of local convergence of bounded degree graphs introduced by Benjamini and Schramm, a sequence of (colored) graphs with maximum degree at most $d$ converges if, for every integer $r$, the distribution of the isomorphism type of the ball of radius $r$ rooted at a random vertex (drawn uniformly at random) converges. 
It is easily checked that in the context of (colored) graphs with bounded degrees, this is equivalent to ${\rm FO}^{\rm local}$-convergence \cite{CMUC}, which justifies using the name of ``local convergence'' for both notions.
The limit object of a local convergent sequence of graphs is a {\em graphing}, that is a graph on a standard Borel space, which satisfies a {\em  Mass Transport Principle}, which amounts to say that for every Borel subsets $A,B$ it holds
$$
\int_A{\rm deg}_B(v)\,{\rm d}v=\int_B{\rm deg}_A(v)\,{\rm d}v.
$$
This Mass Transport Principle is  similar to the Finitary Mass Transport Principle introduced in Definition~\ref{def:FMTP}, in that it expresses that the  edges between two sets $A$ and $B$ can be equally measured from $A$ and from $B$.

	An alternative description of a graphing is as follows: a graphing is defined by a finite number of measure preserving involutions $f_1,\dots,f_D$ on a standard Borel space, which define the edges of the graphing as the union of the orbits of size two of 
$f_1,\dots,f_D$. Using a proper edge coloring, every finite graph with maximum degree $d$ can be represented by means of $d+1$ involutions
$f_1,\dots,f_{d+1}$. It is easily seen that a sequence of properly $d$-edge colored graphs is local convergent if and only if the corresponding sequence of finite algebras with function symbols $f_1,\dots,f_d$ is QF-convergent. In this sense, QF-convergence of algebraic structures is  more general than local-convergence of relational structures with bounded degrees. (However, when considering the same type of structures for both modes of convergence, ${\rm FO}^{\rm local}$-convergence is stronger than QF-convergence.)


As we have seen above, the fragments ${\rm QF}$ (of quantifier-free formulas) and ${\rm FO}^{\rm local}$ (of local formulas) have a specific interest, which is also the case of the fragment ${\rm FO}$ of all first-order formulas. The first main result of this paper, which we prove in Section~\ref{sec:lim}, is the construction of a limit object for ${\rm FO}$-convergent sequence of mappings.

\begin{theorem}
\label{thm:folimmap}
Every  ${\rm FO}$-convergent sequence $(\mathbf F_n)_{n\in\bbbn}$ of finite mappings  (with $\lim_{n\rightarrow\infty} |F_n|=\infty$) has a modeling ${\rm FO}$-limit $\mathbf L$, such that
\begin{enumerate}
	\item the probability measure $\nu_{\mathbf L}$ is atomless;
	\item $\mathbf L$ satisfies the finitary mass transport principle;
		\item the complete theory of $\mathbf L$ has the finite model property.
\end{enumerate} 
\end{theorem}

Three conditions appear in this theorem:
\begin{enumerate}
	\item The measure $\nu_{\mathbf L}$ is {\em atomless}  if for every $v\in L$ it holds $\nu_{\mathbf L}(\{v\})=0$. The necessity of this condition is witnessed by the formula $x_1=x_2$, as $\langle x_1=x_2,\mathbf F\rangle=1/|F|$ holds
for every finite mapping $\mathbf F$. This conditions is thus required as soon as we consider ${\rm QF}$-convergence.
	\item The  finitary mass transport principle (see Definition ~\ref{def:FMTP}), which can be 
	rewritten as follows:  for every Borel subsets $X,Y$ of $L$ and every positive integer $k$  it holds 
\begin{align*}
	(\forall v\in Y)\,|f^{-1}(v)\cap X|=k\quad&\Rightarrow\quad \nu_{\mathbf L}(f^{-1}(Y)\cap X)=k\nu_{\mathbf L}(Y),\\
	(\forall v\in Y)\,|f^{-1}(v)\cap X|>k\quad&\Rightarrow\quad \nu_{\mathbf L}(f^{-1}(Y)\cap X)>k\nu_{\mathbf L}(Y).\\
\end{align*}
This condition is thus required (for definable subsets $X$ and $Y$) as soon as we consider local-convergence.
	\item The {\em complete theory} ${\rm Th}(\mathbf L)$ of $\mathbf L$ is the set of all (first-order) sentences satisfied by $\mathbf L$. This complete theory has the 	 {\em finite model property} (FMP) if for every sentence $\theta\in{\rm Th}(\mathbf L)$ (i.e. for every sentence $\theta$ satisfied by $\mathbf L$) there exists a finite mapping $\mathbf F$ that satisfies $\theta$. This is indeed a necessary condition for $\mathbf L$ to be an elementary limit of finite mappings hence necessary as soon as we consider
	${\rm FO}$-convergence.
\end{enumerate}

As a corollary of Theorem~\ref{thm:folimmap} we deduce (by compactness) the following corollaries.
\begin{corollary}
\label{cor:local}
Every  ${\rm FO}^{\rm local}$-convergent sequence $(\mathbf F_n)_{n\in\bbbn}$ of finite mappings  (with $\lim_{n\rightarrow\infty} |F_n|=\infty$) has a modeling ${\rm FO}^{\rm local}$-limit $\mathbf L$, such that
\begin{enumerate}
	\item the probability measure $\nu_{\mathbf L}$ is atomless;
	\item $\mathbf L$ satisfies the finitary mass transport principle.
\end{enumerate} 	
\end{corollary}

\begin{corollary}
\label{cor:QF}
Every  ${\rm QF}$-convergent sequence $(\mathbf F_n)_{n\in\bbbn}$ of finite mappings  (with $\lim_{n\rightarrow\infty} |F_n|=\infty$) has a Borel mapping limit $\mathbf L$ such that
\begin{enumerate}
	\item the probability measure $\nu_{\mathbf L}$ is atomless;
	\item for every Borel subset $X$ of $Z(\mathbf L)$ it holds $\nu_{\mathbf L}(f_{\mathbf L}(X))=\nu_{\mathbf L}(X)$.
\end{enumerate} 	
\end{corollary}
Theorem~\ref{thm:folimmap} is be proved as a combination of general results about limit distributions (stated below as Theorem~\ref{thm:cmuc}, see \cite{CMUC}) and methods developed in \cite{modeling} for the purpose of graph-trees. As a consequence, in the case of mappings we are able to deduce the form of limits.
We shall also be interested in the inverse problem, which aim to determine which objects are limits of finite mappings (for a given type of convergence). It should be noticed that although the inverse problem for left limits of graphs or hypergraphs has been completely solved, the inverse problem for local limits of graphs with uniformly bounded degrees (or equivalently for ${\rm QF}$-limits of algebras with $d$ involutions), which is known as the Aldous--Lyons conjecture \cite{Aldous2006} is wide open. Note that a positive solution to this inverse problem would have far-reaching consequences, by proving that all finitely generated groups are sofic (settling a question by Weiss \cite{Weiss2000}), the direct finiteness conjecture of Kaplansky \cite{kaplansky1972fields} on group algebras, a conjecture of Gottschalk \cite{Gottschalk1973} on surjunctive groups in topological dynamics, the Determinant Conjecture on Fuglede-Kadison determinants, and Connes' Embedding Conjecture for group von Neumann algebras \cite{Connes1976}.

In our (restricted) setting of algebras with a single function symbol, we solve the inverse problem for ${\rm QF}$-limits (Theorem~\ref{thm:invqfmap}). The solution of  the inverse 
problems for local-limits and ${\rm FO}$-limits,  stated as Theorems~\ref{thm:invlocmap} and~\ref{thm:invfomap} in Section~\ref{sec:conc} will be proved in a forthcoming paper \cite{MapApprox}.

\begin{theorem}
\label{thm:invqfmap}
Let $\mathbf L$ be an atomless Borel mapping.

If every measurable subset $X$ of  the set  $Z(\mathbf L)$ of all the cyclic elements of $\mathbf L$ is such that $\nu_{\mathbf L}(f_{\mathbf L}(X))=\nu_{\mathbf L}(X)$ holds, then $\mathbf L$ is the ${\rm QF}$-limit of a ${\rm QF}$-convergent sequence $(\mathbf F_n)_{n\in\bbbn}$ of finite mappings.
\end{theorem}

Theorem~\ref{thm:invqfmap} is proved in Section~\ref{sec:approx} by a mix of three techniques:

\begin{enumerate}
	\item sampling, 
\begin{center}
\begin{minipage}{.35\textwidth}
	\includegraphics[width=\textwidth]{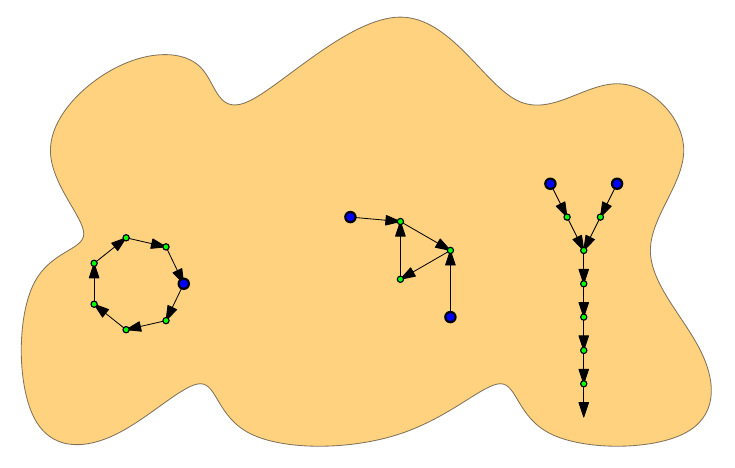}
\end{minipage}
$\longrightarrow$
\begin{minipage}{.35\textwidth}
	\includegraphics[width=\textwidth]{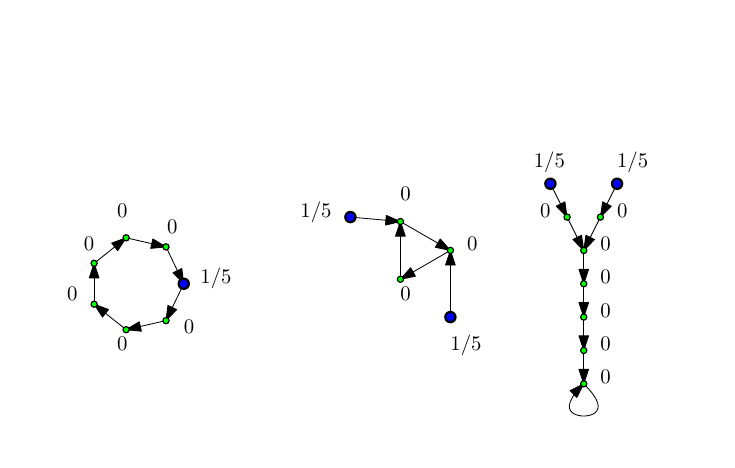}
\end{minipage}
\end{center}
	\item uniformization, 
\begin{center}
\begin{minipage}{.35\textwidth}
	\includegraphics[width=\textwidth]{QF1}
\end{minipage}
$\longrightarrow$
\begin{minipage}{.35\textwidth}
	\includegraphics[width=\textwidth]{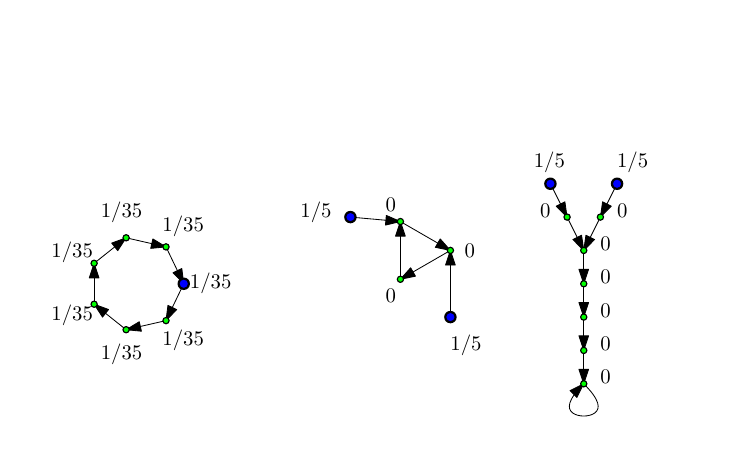}
\end{minipage}
\end{center}
	\item and blowing.
\begin{center}
\begin{minipage}{.35\textwidth}
	\includegraphics[width=\textwidth]{QF2}
\end{minipage}
$\longrightarrow$
\begin{minipage}{.35\textwidth}
	\includegraphics[width=\textwidth]{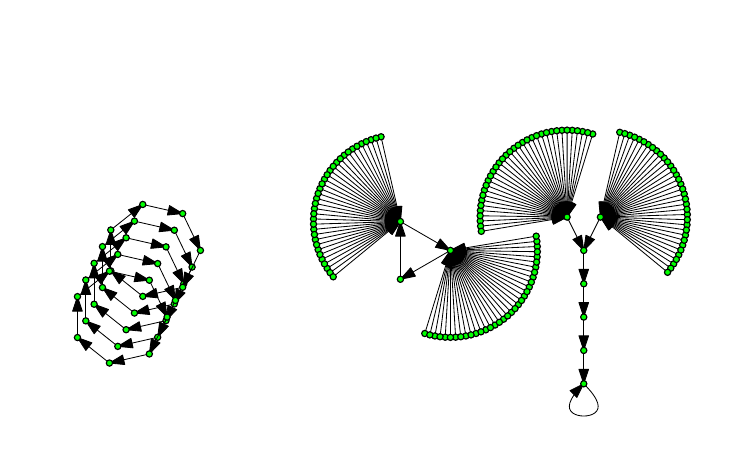}
\end{minipage}
\end{center}
\end{enumerate}

\section{Limits of Mappings}
\label{sec:lim}
For general structures, we have the following general analytic representation theorem, which can be seen as an extension of the representation of left limits of dense graphs by infinite exchangeable graphs \cite{AldousICM,Hoover1979,Kallenberg2005} and of  local limits of graphs with bounded degrees by unimodular distributions \cite{Benjamini2001}:
\begin{theorem}[\cite{CMUC}]
\label{thm:cmuc}
Let $S$ be the Stone dual of the Lindenbaum-Tarski algebra $\mathcal B$ defined by a fragment $X$ of first-order formulas, and let $\Gamma$ be the group of automorphisms of $\mathcal B$ generated by the permutations of the free variables (which naturally acts on $S$). 
For each formula $\phi\in X$, we denote by $I_\phi$ the indicator function of the clopen subset of $S$ dual to $\phi$.

To each finite $\sigma$-structure $\mathbf{A}$ corresponds (injectively) a $\Gamma$-invariant probability measure $\mu_{\mathbf{A}}$ on $S$ such that for every formula $\phi\in X$ it holds
$$
\langle\phi,\mathbf A\rangle=\int_S I_\phi(T)\,{\rm d}\mu_{\mathbf{A}}(T),
$$

For every sequence
$(\mathbf{A}_n)_{n\in\Nat}$ of $\sigma$-structures the sequence $(\mathbf{A}_n)_{n\in\Nat}$ is $X$-convergent if and only if the measures 
$\mu_{\mathbf{A}_n}$ converge weakly. Moreover,
if the sequence $(\mathbf{A}_n)_{n\in\Nat}$ is $X$-convergent
 then the measures 
$\mu_{\mathbf{A}_n}$ converge
to some $\Gamma$-invariant probability measure $\mu$ with the property that for every formula $\phi\in X$ it holds
$$
\lim_{n\rightarrow\infty}\langle\phi,\mathbf{A}_n\rangle=\int_S I_\phi(T)\,{\rm d}\mu(T).
$$
\end{theorem}

We take time to comment on Theorem~\ref{thm:cmuc}. Given a fixed finite (or countably infinite) signature $\sigma$, the {\em Stone dual} $S$ of the Lindenbaum-Tarski Boolean algebra defined by a fragment $X$ of first-order formulas in the language of $\sigma$ is a Polish space (that is a separable completely metrizable topological space) $S$, such that
\begin{itemize}
	\item a point $\mathbf t$ of $S$ is a maximal consistent sets of formulas in $X$;
	\item a basis of the topology of $S$ is given by the family of all clopen sets $K(\phi)$  (for $\phi\in X$), where $K(\phi)$ is the set of all points $\mathbf t$ containing the formula $\phi$ (that is $K(\phi)$ is the clopen subset of $S$ dual to $\phi$).
\end{itemize}

The Stone duals associated to specific fragments of first-order logic can be given a simple description (see \cite{CMUC}):
\begin{itemize}
	\item If $X={\rm QF}$ then $S$ is the space of all $\sigma$-structures with (same) domain $\mathbb N$. In this case, $\Gamma$ is the permutation group $S_\omega$, whose action on $S$ follows from its action on $\mathbb N$, and $\Gamma$-invariant probability measures on $S$ are called {\em exchangeable random $\sigma$-structures}.
	\item if $X={\rm FO}_1^{\rm local}$ then $S$ is the space of all (elementary equivalence classes of)  connected countable rooted $\sigma$-structures. In this case, $\Gamma$ is trivial.
\end{itemize}

It is not known in general which (full) ${\rm FO}$-limits can be represented by measurable structures built on a standard probability space, in a similar way that the local limit of graphs with bounded degrees can be represented by means of a  graphing. Some partial results are known and it has been conjectured by the authors that such a representation exists for every convergent sequence of graphs within a nowhere dense class \cite{limit1}. (See also \cite{Kardos2017150} for such a representation for limits of matroids.)

%

We shall make use of results of \cite{modeling}, which we state as
Theorem~\ref{thm:reduc}, Lemma~\ref{lem:tree1}, and Lemma~\ref{lem:tree2}, after introducing some preliminary definitions. 

Given a finite signature $\sigma$, let $\sigma^+=\sigma\cup\{R\}$, where $R$ is a unary symbol.
Following \cite{modeling}, 
a {\em rooted $\sigma$-structure} is a $\sigma^+$-structure with exactly one element, called the {\em root},  marked by $R$ (that is: with exactly one element in the relation $R$). A rooted $\sigma$-structure  is thus defined by a $\sigma$-structure $\mathbf A$ and the unique element $\rho\in A$ marked by $R$, justifying the notation $(\mathbf A,\rho)$ for the $\sigma$-structure $\mathbf A$ rooted at $\rho$.

\begin{definition}
A sequence $(\mathbf A_n)_{n\in\bbbn}$ of structures is {\em residual} if
$$\forall r\in\bbbn,\quad \limsup_{n\rightarrow\infty}\sup_{v\in A_n}\nu_{\mathbf A_n}(B_r(\mathbf A_n,v))=0.$$

A sequence $(\mathbf A_n,\rho_n)_{n\in\bbbn}$ of rooted structures is {\em $\rho$-non-dispersive} if
$$\forall\epsilon>0\,\exists d\in\bbbn,\quad \liminf_{n\rightarrow\infty}\nu_{\mathbf A_n}(B_d(\mathbf A_n,\rho_n))>1-\epsilon.$$
\end{definition}

\begin{remark}
	Rooted $\sigma$-structures could alternatively be defined by adding a constant (defining the root) to the signature. The assumption of a root marked by a unary relation is weaker in the sense that a fragment of first-order logic like ${\rm FO}^{\rm local}$ does not allow, in general, to locate the root (it can be far from the free variables). However, in the case of $\rho$-non-dispersive sequences, the root is asymptotically almost surely close to the free variables and the two notions are (almost) equivalent.
\end{remark}

In order to prove Theorem~\ref{thm:folimmap}
we make use of the following reduction theorem.

\begin{theorem}[\cite{modeling}]
\label{thm:reduc}
Let $\mathcal C$ be a hereditary class of relational structures, that is a class of relational structures closed under taking induced substructures.

Assume that for every $\mathbf A_n\in\mathcal C$ and every $\rho_n\in A_n$  ($n\in\bbbn$) the following
properties hold:
\begin{enumerate}
	\item  if $(\mathbf A_n)_{n\in\bbbn}$ is  ${\rm FO}_1^{\rm local}$-convergent and residual, then it has
a modeling ${\rm FO}_1^{\rm local}$-limit;
\item  if $(\mathbf A_n,\rho_n)_{n\in\bbbn}$ is ${\rm FO}^{\rm local}$-convergent and
$\rho$-non-dispersive  then it has
a modeling ${\rm FO}^{\rm local}$-limit.
\end{enumerate}
Then $\mathcal C$ admits ${\rm FO}$-modeling limits.
Moreover, if in cases (1) and (2) the modeling limits satisfy the Strong Finitary Mass Transport Principle\footnote{In the context of (rooted) forest modelings, the notion of Strong Finitary Mass Transport Principle used in \cite{modeling} is the graph analog of the Finitary Mass Transport Principle we introduced for mappings in Definition~\ref{def:FMTP}.}, then $\mathcal C$ admits ${\rm FO}$-modeling limits  that satisfy the Strong Finitary Mass Transport Principle.
\end{theorem}

Using this theorem we will be able either to use directly or to mimic the proof of the construction of modeling FO-limits for FO-convergent sequences of  forests, and more specifically to the following two lemmas.

\begin{lemma}[{\cite[Lemma 40]{modeling}}]
\label{lem:tree1}
	Every ${\rm FO}_1^{\rm local}$-convergent residual sequence of forests
$(Y_n)_{n\in\bbbn}$ has a modeling ${\rm FO}_1^{\rm local}$-limit (that satisfies the Strong Finitary Mass Transport Principle).
\end{lemma}
\begin{lemma}[{\cite[Lemma 41]{modeling}}]
\label{lem:tree2}
Every  ${\rm FO}$-convergent $\rho$-non-dispersive sequence of rooted trees $(Y_n)_{n\in\bbbn}$ has a modeling ${\rm FO}$-limit (that satisfies the Strong  Finitary Mass Transport Principle).
\end{lemma}

In order to transport constructions from a type of structures (trees, rooted trees) to another one (mapping, rooted mapping), we shall make use of a powerful technique (see for instance \cite{limit1}), which draw strength from the model theoretic notion of interpretation.  We shall make use here of the simplest instance of it.
Let $\kappa,\sigma$ be signatures, where $\sigma$ has 
$q$ relational symbols $R_1,\dots,R_q$ with respective arities $r_1,\dots,r_q$.
A {\em simple interpretation scheme} ${\mathsf I}$ of $\sigma$-structures
in $\kappa$-structures is defined by 
 a first-order formula $\theta_i$ with $r_i$ free variables (in the language of $\kappa$-structures) for each symbol $R_i\in\sigma$.
 Each simple interpretation scheme $\mathsf I$ defines two mappings, which we also denote by $\mathsf I$ by standard abuse of notations:
 \begin{itemize}
 	\item A mapping from $\kappa$-structures to $\sigma$-structures, defined as follows: for a $\kappa$-structure $\mathbf A$, the $\sigma$-structure $\mathsf I(\mathbf A)$ has same domain as $\mathbf A$, and is such that for every $R_i\in\sigma$ with arity $r_i$ and every $(v_1,\dots,v_{r_i}\in A$ it holds
 	\[
 	\mathsf I(\mathbf A)\models R(v_1,\dots,v_{r_i})\quad\iff\quad
 	\mathbf A\models \theta_i(v_1,\dots,v_{r_i}).
 	\]
 	\item A mapping from first-order formulas with $p$ free variables in the language of $\sigma$-structures to first-order formulas with $p$ free variables in the language of $\kappa$-structures, obtained by replacing each occurence of symbol $R_i$ by the corresponding formula $\theta_i$. As a consequence, for every $\kappa$-structure $\mathbf A$, every formula $\varphi$  with $p$ free variables in the language of $\sigma$-structures, and every $v_1,\dots,v_p\in A$ it holds
 	\[
 	\mathsf I(\mathbf A)\models \varphi(v_1,\dots,v_{r_i})\quad\iff\quad
 	\mathbf A\models \mathsf I(\varphi)(v_1,\dots,v_{r_i}).
 	\]
 \end{itemize}
 It is immediate that if a sequence $(\mathbf A_n)_{n\in\mathbb N}$ of $\kappa$-structures is FO-convergent then so is the sequence $(\mathsf I(\mathbf A_n))_{n\in\mathbb N}$. Moreover, if $\mathbf L$ is a modeling FO-limit of the sequence  $(\mathbf A_n)_{n\in\mathbb N}$ then $\mathsf I(\mathbf L)$ (with associated probability measure $\nu_{\mathsf I(\mathbf L)}=\nu_{\mathbf L}$ is a modeling  FO-limit of the sequence $(\mathsf I(\mathbf A_n))_{n\in\mathbb N}$  \cite{limit1}.

\subsection*{Proof of Theorem~\ref{thm:folimmap}}
We encode our mapping as a directed graph and remove directed loops.
This way, mappings are bijectively mapped to loopless directed graphs such that every vertex has outdegree at most one. This class of directed graphs is hereditary, i.e.  closed under taking induced directed subgraphs. 
According to Theorem~\ref{thm:reduc},  we can reduce the problem of existence of modeling limits to two particular cases, namely the case of $\rho$-non-dispersive sequences and the case of residual sequences.

Let $(\mathbf G_n,\rho_n)_{n\in\bbbn}$ be an ${\rm FO}^{\rm local}$-convergent  $\rho$-non-dispersive sequence of directed graphs such that the maximum outdegree is at most $1$. As the sequence is $\rho$-non-dispersive, we can assume that the graphs $G_n$ are connected. If there exists a vertex without an outgoing edge (note that this vertex has to be unique) or if the vertices in the directed cycle are at a distance from $\rho_n$ that grows to infinity then the result follows from the existence of a modeling limit for FO-convergent $\rho$-non-dispersive sequences rooted trees (Lemma~\ref{lem:tree2}). Otherwise, all (but finitely many) $G_n$'s contain a unique directed cycle of length $k$ at bounded distance from $\rho_n$ (this length has to stabilize because of local convergence). For each $n$ we remove an edge in the directed cycle and mark the two vertices to which the removed arc was incident. We extract an {\rm FO}-convergent subsequence and construct an FO-modeling limit of this subsequence (by Lemma~\ref{lem:tree2}). Then we add the edge back using a simple interpretation scheme. This settles the $\rho$-non-dispersive case.

Let $(\mathbf G_n)_{n\in\bbbn}$ be an ${\rm FO}_1^{\rm local}$ residual sequence of directed graphs such that the maximum outdegree is at most $1$.
By Theorem~\ref{thm:cmuc}, the limit is represented by a probability measure $\mu$ on the Stone dual of ${\rm FO}_1^{\rm local}$. Note that each point $\mathbf t$ in the support of $\mu$ is a maximal consistent set of local formulas with a single free variable (see comments after Theorem~\ref{thm:cmuc}). We transform each such point $\mathbf t$ into the complete theory of a rooted connected directed graphs $(\vec{G},\rho)$ with maximum outdegree $1$ (that is into the set of all sentences satisfied by $(\vec G,\rho)$) as follows: a formula $\phi(x_1)\in\mathbf t$ and transform it into the sentence 
$(\exists z)\ (R(z)\wedge\phi(z))$ (where $R$ is the unary relation marking the root). Because of the locality of the formulas in $\mathbf t$, the constructed sentences exactly describe the first-order 
properties of the connected component of the root.
Hence we can reinterpret  $\mu$ as a probability measure on complete theories of rooted connected directed graphs with maximum outdegree $1$.
We partition the support of $\mu$ according to existence (and length) of a directed cycle and existence of a sink. If there is a fixed point, then the result follows from the study of limits of rooted forests (Lemma~\ref{lem:tree1}). For directed cycles of length $k$ this follows also after easy interpretation. It remains the case where every vertex has outdegree $1$ and where the connected component contains no directed cycle, which is an easy extension of Lemma~\ref{lem:tree1}.
	The construction of ${\rm FO}$-limits of rooted forests in the proof of Lemma~\ref{lem:tree1} \cite[Section 8]{modeling} only uses the assumption that the limit is acyclic and has a ``parent'' mapping (that is a function mapping a vertex to its neighbour closer to the root). Thus, it follows that we can apply it directly to residual limits of asymptotically acyclic mappings. For the connected components of the limit with a circuit of length $k$, the limit is obtained by using a simple interpretation scheme.

Note that the probability measure associated to the constructed modeling FO-limit is atomless (because the limit probability of satisfaction of the formula $x_1=x_2$ is $0$), it satisfies the finitary mass transport principle (because the  directed graph modeling we constructed satisfies the analog strong finitary mass transport principle, and its complete theory has the finite model property (as every sentence satisfied by the limit is satisfied by all mappings in the sequence from a sufficient large index, by the definition of FO-convergence).
	\qed 

\begin{remark}
Note that the condition that a father mapping can be defined is essential. The construction in \cite{modeling} does not extend to the general case of convergent sequences of high-girth graphs, although all the connected components of the limit are trees.
\end{remark}

If a sequence $(\mathbf F_n)_{n\in\bbbn}$ is only 
 ${\rm FO}^{\rm local}$-convergent (resp. ${\rm QF}$-convergent), one 
 can consider an ${\rm FO}$-convergent subsequence of  $(\mathbf F_n)_{n\in\bbbn}$, and a modeling ${\rm FO}$-limit $\mathbf L$ of this subsequence (which exists, according to Theorem~\ref{thm:folimmap}). 
 This modeling $\mathbf L$ is then obviously a modeling ${\rm FO}^{\rm local}$-limit (resp. a modeling ${\rm QF}$-limit of $(\mathbf F_n)_{n\in\bbbn}$. Thus the next two corollaries follow.
 
\begin{mcorollary}[Corollary~\ref{cor:local}]
Every  ${\rm FO}^{\rm local}$-convergent sequence $(\mathbf F_n)_{n\in\bbbn}$ of finite mappings  (with $\lim_{n\rightarrow\infty} |F_n|=\infty$) has a modeling mapping limit $\mathbf L$, such that
\begin{enumerate}
	\item the probability measure $\nu_{\mathbf L}$ is atomless;
	\item $\mathbf L$ satisfies the finitary mass transport principle.
\end{enumerate} 	
\end{mcorollary}

If $\mathbf F$ is a Borel mapping then
all the sets $Z_k(\mathbf F)$ (thus also $Z(\mathbf F)$) are Borel subsets of $F$.
	Thus for every subset $X$ of $Z(\mathbf F)$ it holds
	$$f_{\mathbf F}(X)=\bigcup_{k\in\bbbn} f_{\mathbf F}(X\cap Z_k(\mathbf F))=\bigcup_{k\in\bbbn} f_{\mathbf F}^{-(k-1)}(X\cap Z_k(\mathbf F))\cap Z_k(\mathbf F).$$
	Hence if $\mathbf F$ is Borel, the image $f_{\mathbf F}(X)$ of a Borel subset $X$ of $Z(\mathbf F)$ is Borel.
	Moreover, if $\mathbf F$ is a modeling mapping that satisfies the finitary mass transport principle, then it follows immediately  that for every measurable subset $X$ of $Z(\mathbf F)$ it holds $\nu_{\mathbf F}(f_{\mathbf F}(X))=\nu_{\mathbf F}(X)$. Hence the following corollary.

\begin{mcorollary}[Corollary~\ref{cor:QF}]
Every  ${\rm QF}$-convergent sequence $(\mathbf F_n)_{n\in\bbbn}$ of finite mappings  (with $\lim_{n\rightarrow\infty} |F_n|=\infty$) has a Borel mapping limit $\mathbf L$ such that
\begin{enumerate}
	\item the probability measure $\nu_{\mathbf L}$ is atomless;
	\item for every Borel subset $X$ of $Z(\mathbf L)$ it holds $\nu_{\mathbf L}(f_{\mathbf L}(X))=\nu_{\mathbf L}(X)$.
\end{enumerate} 	
\end{mcorollary}

\begin{figure}
	\begin{center}
		\includegraphics[width=\textwidth]{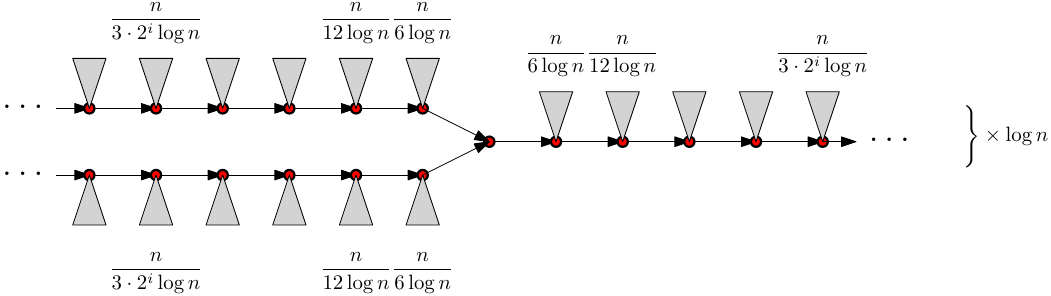}
	\end{center}
	\caption{Example of a sequence of mappings such that in the limit a random vertex belongs with high probability to a connected components with $3$ ends. (Gray triangles represent many elements mapping to a single element.)}
\end{figure}

\begin{figure}
	\begin{center}
		\includegraphics[width=\textwidth]{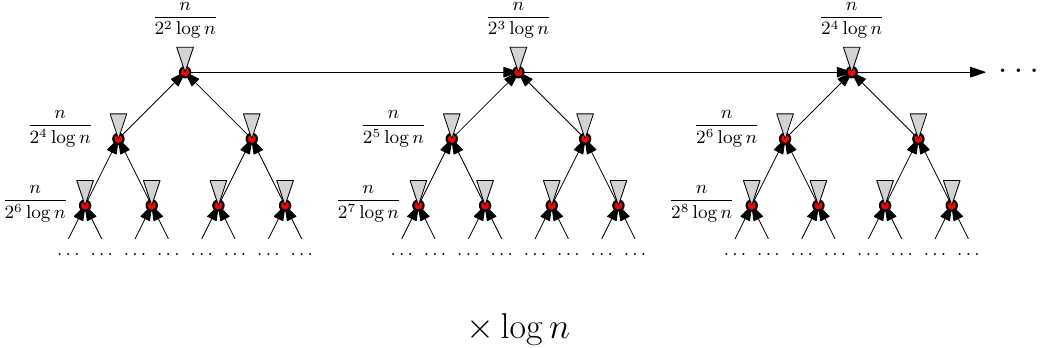}
	\end{center}
	\caption{Example of a sequence of mappings such that in the limit a random vertex belongs with high probability to a connected components with $\omega$ ends. (Gray triangles represent many elements mapping to a single element.)}
\end{figure}

\section{Quantifier-free Approximation}
\label{sec:approx}
The aim of this section is to prove Theorem~\ref{thm:invqfmap}, which states that every atomless Borel mapping $\mathbf L$ such that for every measurable subset $X$ of $Z(\mathbf L)$ it holds $\nu_{\mathbf L}(f_{\mathbf L}(X))=\nu_{\mathbf L}(X)$ is the 
${\rm QF}$-limit of a ${\rm QF}$-convergent sequence of finite mappings.
(Note that in our proofs we only need that $\nu_{\mathbf L}(f_{\mathbf L}(X))=\nu_{\mathbf L}(X)$ holds for Borel subsets $X$ of the form $\phi(\mathbf L)$, where $\phi$ is quantifier-free.)

The proof follows three steps:
\begin{enumerate}
	\item Using Lemma~\ref{lem:samp} below (which was derived in \cite{QFTSL-arxiv} from a concentration result of McDiarmid \cite{mcdiarmid1989method}) we construct an approximation that is a finite mapping on a weighted set.
	\item Then we prove that the weighting can be made constant on cycles, by using a new uniformization technique; this technique relies on the notion of QF-definable group.
	\item Then we show how to blow the elements of the finite weighted model to obtain an unweighted finite approximation.
\end{enumerate}

\begin{lemma}[Corollary 1 of \cite{QFTSL-arxiv}]
\label{lem:samp}
For every signature $\sigma$,  
 every Borel $\sigma$-structure $\mathbf{A}$, every $\epsilon>0$ and every $p,q\in\bbbn$ there exists a finite weighted $\sigma$-structure  $\mathbf B$ such that for every 
quantifier-free formula $\phi$ with at most $p$ free variables and $q$ functional symbols it holds
$$
|\langle\phi,\mathbf{A}\rangle-\langle\phi,\mathbf{B}\rangle|<\epsilon.
$$
\end{lemma}

We define the ternary operator `$\switch{}{}{}$', which is usual in computer science, as follows: for every formula $\phi$ and every terms $t_1,t_2$

$$
\switch{\phi}{t_1}{t_2}=\begin{cases}
	t_1&\text{if }\phi\text{ holds,}\\
	t_2&\text{otherwise.}
\end{cases}
$$

\begin{definition}
	The notion of {\em QF-definable} function is defined inductively:
	\begin{itemize}
		\item every function in the signature is QF-definable;
		\item the identity is QF-definable;
		\item every composition of QF-definable functions is QF-definable;
		\item for every QF-formula $\phi$ and every QF-definable functions $f,g$, the function 
		 $h:=\switch{\phi}{f}{g}$
 is QF-definable.
	\end{itemize}
\end{definition}

For instance, if we consider the signature formed by a single unary function symbol $f$, an example of QF-definable function is
$$
g(x)
=\begin{cases}
	f(x)&\text{if }f^2(x)=x\\
	x&\text{otherwise}
\end{cases}
$$
(as $g(x)=\switch{(f(f(x))=x)}{f(x)}{x}$).

Note that (by easy induction) for every quantifier-free formula $\phi$ (with $p$ free variables) and every $p$-tuple
$\mathbf g=(g_1,\dots,g_p)$ of QF-definable functions  there exists a quantifier-free formula
$\phi_{\mathbf g}$ such that 
$\phi(g_1(x_1),\dots,g_p(x_p))$ is logically equivalent to $\phi_{\mathbf g}(x_1,\dots,x_p)$.

\begin{definition}
	A {\em QF-definable group} is a set $\Gamma$ of QF-definable unary functions, which contains the identity, is closed under compositions, and such that every function in $\Gamma$ has an inverse in $\Gamma$.
\end{definition}

Note that a QF-definable group naturally acts on every $\sigma$-structure.

\begin{lemma}
\label{lem:QFreg}
	Let $\mathbf L$ be a Borel $\sigma$-structure, and let $\Gamma$ be a finite QF-definable group acting on $\sigma$-structures.
	
	If $\nu_{\mathbf L}$ is $\Gamma$-invariant, then for every $\epsilon>0$, and every $p,q\in\bbbn$ there exists a finite weighted $\sigma$-structure $\mathbf F$ with $\Gamma$-invariant $\nu_{\mathbf F}$ such that
	$$
	|\langle\phi,\mathbf F\rangle-\langle\phi,\mathbf L\rangle|<\epsilon
	$$
	holds for every QF-formula with at most $p$ free variables and $q$ functional symbols.
\end{lemma}
\begin{proof}
	
	Let $\phi$ be a quantifier free formula with $k\leq p$ free variables and at most $\ell\leq q$ functional symbols.
	As $\Gamma$ is QF-definable, for every $\mathbf g\in\Gamma^k$ there exists a QF formula $\phi_{\mathbf g}$ such that
	$\phi_{\mathbf g}(x_1,\dots,x_k)=\phi(g_1(x_1),\dots,g_k(x_k))$.
	Moreover, there is a constant $C$ (independent of $\phi$, but dependent of $\Gamma$) such that
	$\phi_{\mathbf g}$ can be required to use at most $Ck\ell$ functional symbols.
	
	According to Lemma~\ref{lem:samp} there exists a finite weighted structure $\mathbf F_0$ such that for every QF formula $\psi$ with at most $p$ free variables and $Cpq$ functional symbols it holds
	$$|\langle\psi,\mathbf F_0\rangle-\langle\psi,\mathbf L\rangle|<\frac{\epsilon}{|\Gamma|^p}.$$

	Let $\mathbf F$ be the structure $\mathbf F_0$ with new probability measure
	$$\nu_{\mathbf F}(\{v\})=\frac{1}{|\Gamma|}\sum_{g\in\Gamma}\nu_{\mathbf F_0}(g(v)).$$

	As $\nu_{\mathbf L}$ is $\Gamma$-invariant, it follows that
$$\langle\phi,\mathbf L\rangle=\frac{1}{|\Gamma|^k}\sum_{\mathbf g\in\Gamma^k}\langle\phi_{\mathbf g},\mathbf L\rangle.$$

As each $\phi_{\mathbf g}$ has at most $p$ free variables and 
$Cpq$  functional symbols, by union bound it holds
	$$
\Bigl|\langle\phi,\mathbf L\rangle-\frac{1}{|\Gamma|^k}\sum_{\mathbf g\in\Gamma^k}\langle\phi_{\mathbf g},\mathbf F_0\rangle\Bigr|<\epsilon.$$
By construction, it holds
$$\sum_{\mathbf g\in\Gamma^k}\langle\phi_{\mathbf g},\mathbf F_0\rangle=\sum_{\mathbf g\in\Gamma^k}\langle\phi_{\mathbf g},\mathbf F\rangle.$$
	As $\nu_{\mathbf F}$ is $\Gamma$-invariant, it follows that
$$\langle\phi,\mathbf F\rangle=\frac{1}{|\Gamma|^k}\sum_{\mathbf g\in\Gamma^k}\langle\phi_{\mathbf g},\mathbf F\rangle.$$
Altogether we get
$$
	|\langle\phi,\mathbf F\rangle-\langle\phi,\mathbf L\rangle|<\epsilon
$$
\end{proof}

We obtain Theorem~\ref{thm:invqfmap} if we apply this lemma to Borel mappings.
\subsection*{Proof of Theorem \ref{thm:invqfmap}}
We shall prove here that for every $\epsilon>0$ and every positive integers $p,q$ there exists a finite mapping $\widehat{\mathbf F}$ (with uniform probability measure $\nu_{\widehat{\mathbf F}}$) such that for every QF-formula $\phi$  with at most $p$ free variables and $q$ function symbols it holds
\[
|\langle\phi,\widehat{\mathbf F}\rangle-\langle\phi,{\mathbf L}\rangle|<\epsilon.
\]

	For $1\leq k\leq q$, let $\zeta_k$ be the QF-definable function
	$$
	\zeta_k(v)=\switch{\Bigl((f^k(v)=v)\wedge\bigwedge_{1\leq i<k}(f^i(v)\neq v)\Bigr)}{f(v)}{v}.
	$$
Note that $\zeta_k^{-1}=\zeta^{k-1}$ hence the group $\Gamma$ generated by $\zeta_1,\dots,\zeta_q$ is a (finite Abelian) QF-definable group. By 
Lemma~\ref{lem:QFreg} there exists a finite weighted mapping $\mathbf F$ such that $\nu_{\mathbf F}$ is $\Gamma$-invariant, and for every
QF-formula $\phi$ with at most $p$ free variables and $q$ functional symbols it holds $|\langle\phi,\mathbf F\rangle-\langle\phi,\mathbf L\rangle|<\epsilon/3$.

Let $N$ be an integer such that
$N>\frac{3}{\epsilon}\Bigl(\frac{p^2(q+1)}{|F|}+2p\Bigr)$.

 Let $v_1,\dots,v_n$ be the elements of $F$, and let $\nu(v_1),\dots,\nu(v_n)$ be their weights.
 We define $\widehat F\subseteq F\times \{0,1\dots,N|F|\}$ by
 $$\widehat F=\{(v,j):\ v\in F, j\in \{0,1,\dots,\lfloor N\nu_{\mathbf F}(\{v\})|F|\rfloor\}\}$$
and the mapping $\widehat{\mathbf F}$ with domain $\widehat F$
and uniform probability measure $\nu_{\widehat{\mathbf F}}$ 
 by
$$f_{\widehat{\mathbf F}}(v,j)=
\begin{cases}
(f_{\mathbf F}(v),j)&\text{if }v\in \bigcup_{i=0}^qZ_i(\mathbf F),\\
(f_{\mathbf F}(v),0)&\text{otherwise}. 
\end{cases}$$
Note that $N |F|\leq |\widehat{F}|<(N+1)|F|$.

Let $\phi$ be a quantifier-free formula with $p$ free variables and $q$ function symbols, and let $X_1,\dots,X_p$ be random elements of $\widehat F$ drawn uniformly and independently at random (that is with respect to the uniform probability measure $\nu_{\widehat{\mathbf F}}$).
Let $\pi:\widehat F\rightarrow F$ be the projection $(v,j)\mapsto v$, and let  $\tilde{\nu}$ be the push-forward of the measure $\nu_{\widehat{\mathbf F}}$ by $\pi$.
Then $\pi(X_1),\dots, \pi(X_p)$ are i.i.d. elements of $F$ distributed with respect to the probability measure $\tilde{\nu}$. 
So let $\widetilde{\mathbf F}$ be the mapping $\mathbf F$ with probability measure $\tilde{\nu}$ instead of $\nu_{\mathbf F}$
(that is: $\nu_{\widetilde{\mathbf F}}=\tilde{\nu}$).

Note that for every $v\in F$ it holds
$$\nu_{\mathbf F}(\{v\})-\frac{1}{|\widehat{F}|}<\tilde{\nu}(\{v\})\leq \nu_{\mathbf F}(\{v\}).$$
Thus 
\begin{align*}
	|\langle\phi,\widetilde{\mathbf F}\rangle-\langle\phi,\mathbf F\rangle|&\leq \nu_{\mathbf F}^{\otimes p}(\phi(\mathbf F))-\tilde{\nu}^{\otimes p}(\phi(\mathbf F))\\
	&\leq p\sup_{X\subseteq F} |\nu_{\mathbf F}(X)-\tilde{\nu}(X)|\\
	&\leq p |F|/|\widehat F|\leq p/N.
\end{align*}

Let 
$\xi$ be the formula
$$
\xi(x_1,\dots,x_p):=\bigwedge_{\substack{1\leq i\neq j\leq p\\0\leq a\leq q}} f^a(x_i)\neq x_j,
$$
where $f^0$ is the identity function.

As $\mathbf L$ is atomless, it holds $\langle\xi,\mathbf L\rangle=0$.
Also, it is easily checked that $\langle\xi,\widehat{\mathbf  F}\rangle\geq 1-\binom{p}{2}(q+1)/|\widehat F|$.

Let $\psi=\phi\wedge\xi$.
Then $\psi$ can be written as 
\[\psi=\xi\wedge\bigvee\!\!\!\!\!\!\bigvee_{\alpha}\bigwedge_\beta \phi_{\alpha,\beta},\] 
where  $\vee\!\!\!\vee$ denotes an exclusive disjunction, and where  the literals $\phi_{\alpha,\beta}$'s are   of the form
$f^a(x_i)=f^b(x_j)$ or $f^a(x_i)\neq f^b(x_j)$ (with no $(i,j,a,b)$ of the form $(i,j,a,0)$ for some  $i\neq j$). 

There are two cases:
\begin{itemize}
\item  $(i,j,a,b)$ is such that $i\neq j$ and $a,b>0$.
 
 For every $(u,k),(v,\ell)\in \widehat{F}$ and $u\neq v$  it holds \[f_{\widehat{\mathbf F}}^a(u,k)= f_{\widehat{\mathbf F}}^b(v,\ell)\iff f_{\mathbf F}^a(u)= f_{\mathbf F}^b(v).\] 
Thus 
$\langle f^a(x_i)=f^{b}(x_j),\widehat{\mathbf F}\rangle=\langle f^a(x_i)=f^{b}(x_j),\widetilde{\mathbf F}\rangle$.	
\item $(i,j,a,b)$ has form  $(i,i,a,a+b)$ with $b>0$.
 
For every $(u,k)\in \widehat{F}$ it holds
\[f_{\widehat{\mathbf F}}^{a}(u,k)=f_{\widehat{\mathbf F}}^{a+b}(u,k)\iff f_{\mathbf F}^{a}(u)=f_{\mathbf F}^{a+b}(u).\]
 Thus
 $\langle f^a(x_i)=f^{a+b}(x_i),\widehat{\mathbf F}\rangle=
\langle f^a(x_i)=f^{a+b}(x_i),\widetilde{\mathbf F}\rangle$.
\end{itemize}

Altogether, we get that
\begin{align*}
	|\langle\phi, \widehat{\mathbf F}\rangle-\langle\phi, \widetilde{\mathbf F}\rangle|&\leq 
	\langle\neg\xi,\widehat{\mathbf F}\rangle+
	\langle\neg\xi,\widetilde{\mathbf F}\rangle\\
	&\leq \binom{p}{2}(q+1)/|\widehat F|+p/N+\langle\neg\xi,{\mathbf F}\rangle\\
	&\leq \frac{1}{N}\biggl(\frac{p^2(q+1)}{|F|}+p\biggr)+\epsilon/3.
	\intertext{Thus}
	|\langle\phi, \widehat{\mathbf F}\rangle-\langle\phi, \mathbf L\rangle|&\leq |\langle\phi, \widehat{\mathbf F}\rangle-\langle\phi, \widetilde{\mathbf F}\rangle|+|\langle\phi, \widetilde{\mathbf F}\rangle-\langle\phi, {\mathbf F}\rangle|+|\langle\phi, \mathbf F\rangle-\langle\phi, \mathbf L\rangle|\\
	&\leq \frac{1}{N}\biggl(\frac{p^2(q+1)}{|F|}+2p\biggr)+2\epsilon/3.
	\end{align*}
Hence, according to the definition of $N$ it holds
	\[|\langle\phi, \widehat{\mathbf F}\rangle-\langle\phi,\mathbf L\rangle|<\epsilon.\]

\qed 

\begin{remark}
Note that in the above construction, we construct approximations of order $Cn(1+o(1))$ for some constant $C$ and sufficiently large (but arbitrary) $n$.	
\end{remark}

\section{Concluding Remarks}
\label{sec:conc}

Theorem~\ref{thm:invqfmap} can be generalized to other fragments of first-order formulas. This is reflected for example by the following two results.

\begin{theorem}
\label{thm:invlocmap}
Every atomless modeling mapping $\mathbf L$  that satisfies the finitary mass transport principle is the 
${\rm FO}^{\rm local}$-limit of an ${\rm FO}^{\rm local}$-convergent sequence of finite mappings.
\end{theorem}

\begin{theorem}
\label{thm:invfomap}
Every atomless modeling mapping $\mathbf L$ with  finite model property that satisfies the finitary mass transport principle is the 
${\rm FO}$-limit of an ${\rm FO}$-convergent sequence of finite mappings.
\end{theorem}

We do not include the proofs of these results as they are lengthy and increasingly more model theory related. This will appear elsewhere and the interested reader may consult the companion paper \cite{MapApprox}.

So we see that the question of limits is not easy, even for the simplest case of one mapping. It is  an interesting question to try to extend (some of) these results for several mappings, say $f_1,\dots,f_d$ (on a same set). 

For the construction of a limit object there is some hope that our approach could extend to the construction (or at least to a proof of existence) of a Borel limit structure in the case of QF-convergent sequences.

However,  the existence of a modeling limit for FO-convergent sequences of algebras with $d\geq 2$ mappings is ruled out by  \cite[Theorem 27]{modeling}, as it would allow to construct modeling limits for monotone somewhere dense classes of graphs (see \cite{Sparsity} for more about nowhere dense classes and related notions).

Also, although Borel QF-limits may exist for $d=2$, the inverse theorem (which we established for $d=1$ as Theorem~\ref{thm:invqfmap}) is a hard problem as it involves Aldous-Lyons conjecture \cite{Aldous2006}. 
In this respect the above may shed more light on this difficult problem.

\section*{Acknowledgments}
The authors wish to thank to the anonymous referees for their careful reading and for their valuable comments and suggestions, which definitely improved the presentation of this paper.

\section*{References}
\providecommand{\bysame}{\leavevmode\hbox to3em{\hrulefill}\thinspace}
\providecommand{\MR}{\relax\ifhmode\unskip\space\fi MR }
\providecommand{\MRhref}[2]{%
  \href{http://www.ams.org/mathscinet-getitem?mr=#1}{#2}
}
\providecommand{\href}[2]{#2}

\end{document}